\documentclass[12pt]{article}
\usepackage{amsmath}
\usepackage{amssymb}
\usepackage{latexsym}
\usepackage{amsthm}
\usepackage{mathrsfs}
\usepackage{enumitem}
\usepackage[colorlinks,
            linkcolor=red,
            anchorcolor=blue,
            citecolor=green]{hyperref}

\newtheorem{thm}{Theorem}[section]
\newtheorem{lem}[thm]{Lemma}
\newtheorem{prop}[thm]{Proposition}

\newtheorem{conj}[thm]{Conjecture}
\newtheorem{rem}[thm]{Remark}

\newcommand{\e}{\mathbf{e}}
\newcommand{\s}{\mathbf{s}}

\DeclareMathOperator{\asc}{asc}
\DeclareMathOperator{\Asc}{Asc}

\DeclareMathOperator{\des}{des}
\DeclareMathOperator{\Des}{Des}

\DeclareMathOperator{\nege}{neg}
\DeclareMathOperator{\exc}{exc}

\newcommand{\I}{\mathfrak{I}} 
\newcommand{\seps}{\prec}
\newcommand{\sep}{\preceq}

\begin{document}
\begin{center}
{\large \bf  Mutual Interlacing and Eulerian-like Polynomials\\
 for  Weyl Groups}
\end{center}
\begin{center}
Arthur L.B. Yang$^{1}$ and Philip B. Zhang$^{2}$\\[6pt]

$^{1, 2}$Center for Combinatorics, LPMC-TJKLC\\
Nankai University, Tianjin 300071, P. R. China\\[6pt]


Email: $^{1}${\tt yang@nankai.edu.cn},
             $^{2}${\tt zhangbiaonk@163.com}
\end{center}

\noindent\textbf{Abstract.}
We use the method of mutual interlacing to prove two conjectures on the real-rootedness of Eulerian-like polynomials: Brenti's conjecture on $q$-Eulerian polynomials for Weyl groups of type $D$, and Dilks, Petersen, and Stembridge's conjecture on affine Eulerian polynomials for irreducible finite Weyl groups.

For the former, we obtain a refinement of Brenti's $q$-Eulerian polynomials of type $D$, and then show that these refined Eulerian polynomials satisfy certain recurrence relation. By using the Routh--Hurwitz theory and the recurrence relation, we prove that these polynomials form a mutually interlacing sequence for any positive $q$, and hence prove Brenti's conjecture. For $q=1$, our result reduces to the real-rootedness of the Eulerian polynomials of type $D$, which were originally conjectured by Brenti and recently proved by Savage and Visontai.

For the latter, we introduce a family of polynomials based on Savage and Visontai's refinement of Eulerian polynomials of type $D$. We show that these new polynomials satisfy the same recurrence relation as Savage and Visontai's refined Eulerian polynomials.
As a result, we get the real-rootedness of the affine Eulerian polynomials of type $D$. Combining the previous results for other types, we completely  prove Dilks, Petersen, and Stembridge's conjecture, which states that, for every irreducible finite Weyl group, the affine descent polynomial has only real zeros.

\noindent \emph{AMS Classification 2010:} Primary 05A15, 26C10; Secondary 20F55, 05E45, 93D05.


\section{Introduction}

Let $\mathfrak{S}_n$ denote the set of permutations of $[n]=\{1,2,\ldots,n\}$.
For $\sigma=(\sigma_1,\sigma_2, \ldots, \sigma_n) \in \mathfrak{S}_n$, let
$$\Des \sigma = \{i \in [n-1] : \sigma_i > \sigma_{i+1}\}$$
denote the set of descents of $\sigma$, and let $\des \sigma = |\Des \sigma|.$ The Eulerian polynomials $S_n(x)$ are usually defined as the descent generating function over $\mathfrak{S}_n$, namely,
\begin{align}\label{eq-eulerionpol}
S_n(x) = \sum_{\sigma\in\mathfrak{S}_n}x^{\des(\sigma)}.
\end{align}
These polynomials are not only of interest in combinatorics, but also
of significance in geometry. For example, the coefficients of Eulerian polynomials can be interpreted as the $h$-vector of the Coxeter complex of type $A$, or as the even Betti numbers of certain toric varieties, see \cite{Dolgachev1994character, Stanley1980number, Stembridge1992Eulerian}.

There are many interesting generalizations of Eulerian polynomials, see \cite{Brenti1994$q$, Dilks2009Affine, Savage$s$,  Visontai2013Stable} and references therein. In this paper, we focus on two families of Eulerian-like polynomials, which are Brenti's $q$-analogue of Eulerian polynomials for finite Coxeter groups \cite{Brenti1994$q$} and Dilks, Petersen, and Stembridge's affine Eulerian polynomials for irreducible finite Weyl groups \cite{Dilks2009Affine}.
It is well known that the classical Eulerian polynomials have only real zeros.
By the Newton inequality, if a polynomial
$$f(x)=a_0+a_1x+\cdots+a_nx^n$$
has only nonpositive zeros, then its coefficients must be log-concave, namely, $$a_i^2\geq a_{i+1}a_{i-1}, \mbox{ for } 1\leq i\leq n-1.$$
Therefore, the polynomial $f(x)$ is also unimodal, namely, there exists some $k$ such that
$$a_0\leq \cdots \leq a_{k-1}\leq a_k \geq a_{k+1}\geq \cdots \geq a_n.$$
Thus, the Eulerian polynomials $S_n(x)$ are log-concave and unimodal.
Many sequences and polynomials appearing in combinatorics, algebra and geometry turn out to be log-concave or unimodal, see  \cite{Brenti1989Unimodal, Brenti1994Log, Stanley1989Log}.
A natural problem is to study whether the Eulerian-like polynomials are unimodal, log-concave, or even real-rooted. Brenti \cite{Brenti1994$q$} obtained the real-rootedness of the $q$-Eulerian polynomials of type $B$ for any positive $q$, and conjectured that it is also true for type $D$.
Dilks, Petersen, and Stembridge \cite{Dilks2009Affine} showed that the affine Eulerian polynomials for irreducible finite Weyl groups have symmetric and unimodal coefficients, and conjectured these polynomials are real-rooted.
The main objective of this paper is to prove Brenti's conjecture on the real-rootedness of $q$-Eulerian polynomials and Dilks, Petersen, and Stembridge's conjecture on the real-rootedness of affine Eulerian polynomials.

Let us first give an overview of Brenti's conjecture and Dilks, Petersen, and Stembridge's conjecture. We assume that the reader is familiar with Coxeter groups and root systems, see \cite{Bjorner2005Combinatorics, Humphreys1990Reflection}.
Let $W$ be a finite Coxeter group generated by $s_{1},s_{2},\ldots,s_{n}$.
The length of each $\sigma\in W$ is defined as the number of generators in one of its reduced expressions, denoted $\ell(\sigma)$. We say that $i$ is a descent of $\sigma$ if $\ell(\sigma s_{i})<\ell(\sigma)$. Let $\Des\sigma$ denote the descent set of $\sigma$, and let $\des \sigma=|\Des\sigma|$ denote the descent number.
The descent polynomial for a finite Coxeter group $W$ is
defined by
\begin{align*}
W(x)\ =\ \sum_{\sigma\in W}x^{\des\sigma}\,.
\end{align*}
In a geometric context, this polynomial is also the $h$-polynomial of the Coxeter complex of $W$, for more information see \cite{Stembridge1994Some, Stembridge2008Coxeter}.

Brenti \cite{Brenti1994$q$} first studied the problem of whether $W(x)$ has only real zeros for any general finite Coxeter group. By a simple argument, he showed that it is enough to check the real-rootedness of $W(x)$ for irreducible finite Coxeter groups. If $W$ is a Coxeter group of type $A_{n}$ (or $B_n, D_n, \ldots$), by abuse of notation, we shall write the corresponding descent polynomials as $A_n(x)$  (resp. $B_n(x), D_n(x), \ldots$) instead of $W(x)$.
For the exceptional groups, one can directly verify the truth by using a computer. The real-rootedness of $A_{n}(x)$ is also obvious, since it is the Eulerian polynomial $S_{n+1}(x)$ as defined in \eqref{eq-eulerionpol}.

To prove the real-rootedness of $B_n(x)$, Brenti \cite{Brenti1994$q$} gave a combinatorial interpretation of descents in the following manner.
First, regard the Coxeter group $B_{n}$ as the set of signed permutations of the set $[n]$, i.e., each element $\sigma\in B_n$ is a permutation of $\{-n,\ldots,-1,1,\ldots,n\}$ satisfying $\sigma({-i})=-\sigma(i)$ for $1\le i \le n$. Then write $\sigma$ in one-line notation $(\sigma_{1},\sigma_{2},\dots,\sigma_{n})$, where $\sigma_i=\sigma(i)$.
Let $\nege \sigma$ be the negative numbers in $(\sigma_{1},\sigma_{2},\dots,\sigma_{n})$, and let $\des_{B}\sigma =|\Des_B\sigma|$, where
$$\Des_{B}\sigma  = \,\{0 : \mathrm{if}\;\sigma_{1}<0\}\cup\{i\in [n-1] : \sigma_{i}>\sigma_{i+1}\}.$$
Brenti introduced the following $q$-analogue of $B_n(x)$:
\begin{align}
B_n(x;q)= \sum_{\sigma\in{B_n}}q^{\nege\sigma}x^{{\des}_{B}\sigma},
\end{align}
which reduces to $B_n(x)$ when $q=1$. Brenti proved that, for any $q\geq 0$, the polynomial $B_n(x;q)$ has only real zeros, and thus established the real-rootedness of $B_n(x)$.

The real-rootedness conjecture of $D_n(x)$ has resisted all efforts until recently. Analogous to the case of type $B_n$, Brenti gave a combinatorial interpretation of $D_n(x)$ as certain generating function over the set of even signed permutations of $[n]$. Given an even signed permutation $\sigma$ with one-line notation $(\sigma_{1},\sigma_{2},\dots,\sigma_{n})$, let $\nege_D\sigma$ be the negative numbers in $(\sigma_{2},\dots,\sigma_{n})$, and let
$\des_D\sigma=|\Des_D\sigma|$, where
$$\Des_{D}\sigma  =  \,\{0 : \mathrm{if}\;\sigma_{1}+\sigma_{2}<0\}\cup\{i\in [n-1] : \sigma_{i}>\sigma_{i+1}\}.$$
Brenti introduced the following $q$-analogue of $D_n(x)$,
\begin{align}\label{def:q-D}
D_n(x;q)= \sum_{\sigma\in{{D}_n}}q^{\nege_D\sigma}x^{{\des}_{D}\sigma},
\end{align}
which reduces to $A_{n-1}(x)$ when $q=0$.
Brenti made the following conjecture.

\begin{conj}[{\cite{Brenti1994$q$}}]\label{conj:q-type-D}
For any positive $q$, the polynomial $D_n(x;q)$ has only real zeros.
\end{conj}

Note that, when $q=1$, the polynomial $D_n(x;q)$ reduces  to $D_n(x)$.
For this case, Savage and Visontai \cite{Savage$s$} gave a novel proof of the above conjecture by using the theory of $\mathbf{s}$-inversion sequences.
Recall that, given a sequence $\s=(s_{1}, s_2,\ldots)$ of positive integers, an $n$-dimensional $\s$-inversion sequence is a sequence $\e=(e_{1},\dots,e_{n})\in\mathbb{N}^{n}$ such that $e_i<s_i$ for each $1\le i\le n$. Denote the set of $n$-dimensional $\s$-inversion sequences by
$\I_{n}^{(\s)}$. For $\s=(2,4,6,\dotsc)$, Savage and Visontai introduced a statistic $\asc_{D}$ on inversion sequences  $\e=(e_{1},\dots,e_{n})\in\I_{n}^{(\s)}$, which counts the number of type $D$ ascents given by
\begin{align}\label{def:asc-D}
\Asc_{D}\e=\{ i\in[n-1] : \frac{e_{i}}{i}<\frac{e_{i+1}}{i+1}\} \cup\{0 : {\rm if}\ e_{1}+e_{2}/2\ge3/2\}.
\end{align}
In this way, the polynomial $D_n(x)$ can be interpreted as the generating function of the statistic $\asc_{D}$ over $\I_{n}^{(2,4,6,\dotsc)}$, precisely,
\begin{align*}
2D_{n}(x)=\sum_{e\in\I_{n}^{(2,4,6,\dotsc)}}x^{\asc_{D}e}.
\end{align*}
Let $T_n(x)=2D_{n}(x)$. Clearly, $T_n(x)$ is real-rooted if and only if $D_n(x)$ is real-rooted. To prove the real-rootedness of $T_n(x)$, Savage and Visontai introduced the following refinement of $T_n(x)$:
\begin{align}\label{refinement typed}
T_{n,i}(x)\ =\ \sum_{\e\in\I_{n}^{(2,4,6,\dotsc)}}\chi(e_{n}=i)\, x^{\asc_{D}\e}\,,
\end{align}
where $\chi(\varphi)$ is $1$ if the statement $\varphi$ is true and $0$ otherwise. Note that
$$T_n(x)=\sum_{i=0}^{2n-1}T_{n,i}(x).$$ They showed that, for any $n\geq3$ and $0\le i\le 2n-1$, these refined polynomials satisfy the following simple recurrence relation:
\begin{align}\label{eq:T_n,i}
T_{n,i}(x)=x\sum_{j=0}^{\left\lceil \tfrac{n-1}{n}i\right\rceil -1}T_{n-1,j}(x)+\sum_{j=\left\lceil \tfrac{n-1}{n}i\right\rceil }^{2n-3}T_{n-1,j}(x),
\end{align}
where $\lceil t\rceil$ represents the smallest integer larger than or equal to $t$. By using the theory of compatible polynomials developed by
Chudnovsky and Seymour \cite{Chudnovsky2007roots}, Savage and Visontai inductively proved that the polynomials satisfying such recurrence relations are compatible, and hereby obtained the real-rootedness of $T_n(x)$. The basis of their induction is $n=4$, which can be verified by numerical analysis with the aid of a computer. As a result, Savage and Visontai proved the following result, a long-standing conjecture of Brenti.

\begin{thm}[{\cite[Theorem 3.15]{Savage$s$}}]
For any finite Coxeter group $W$, the descent polynomial $W(x)$ has only real zeros.
\end{thm}

Dilks, Petersen, and Stembridge \cite{Dilks2009Affine} proposed a
companion conjecture to Brenti's conjecture, which is concerned with
the real-rootedness of certain affine descent polynomials for  irreducible finite Weyl groups. Suppose that $W$ is an irreducible finite Weyl group generated by $\{s_1,s_2,\ldots,s_n\}$. Let $s_0$ be the reflection corresponding to the highest root.
For each $\sigma\in W$, we say that $i$ is an affine descent of $\sigma$ if either $i\in \Des\,\sigma$ for $1\le i\le n$, or $i=0$ and $\ell(\sigma s_{0})>\ell(\sigma)$.
Let $\widetilde{\Des}\,\sigma$ denote the set of affine descents of $\sigma$, and let $\widetilde{\des}\,\sigma=|\widetilde{\Des}\,\sigma|$. It is worth mentioning that the affine descents were first introduced by Cellini \cite{Cellini1995general} for finite Weyl groups, for further developments see \cite{Cellini1995generala, Cellini1998Cyclic, Fulman2000Affine, LamAlcoved, Petersen2005Cyclic}.
 Analogous to the definition of $W(x)$, the affine descent polynomial of $W$ is defined as
\begin{align*}
\widetilde{W}(x)\ =\ \sum_{\sigma\in W}x^{\widetilde{\des}\,\sigma},
\end{align*}
which is called the affine Eulerian polynomial by Dilks, Petersen, and Stembridge.
They obtained many interesting properties of $\widetilde{W}(x)$, such as a connection with the $h$-polynomial of the reduced Steinberg torus.
They also showed that the affine Eulerian polynomials have unimodal coefficients.
Furthermore, Dilks, Petersen, and Stembridge proposed the following conjecture.

\begin{conj}[{\cite[Conjecture 4.1]{Dilks2009Affine}}]
For any irreducible finite Weyl group $W$, the affine Eulerian polynomial $\widetilde{W}(x)$ has only real zeros.
\label{conj:affine-type-D}
\end{conj}

Dilks, Petersen, and Stembridge \cite{Dilks2009Affine} remarked that the affine Eulerian polynomials $\widetilde{A}_n(x)$ and $\widetilde{C}_n(x)$ are both  multiples of the classical Eulerian polynomial $S_{n}(x)$ and hence the above conjecture is true for the groups of type $A$ and $C$, see also \cite{Fulman2000Affine, Petersen2005Cyclic}.
For the exceptional groups, the conjecture can be directly verified.
Dilks, Petersen, and Stembridge \cite{Dilks2009Affine} left the type $B$ and type $D$ cases open.
Savage and Visontai's novel approach to Brenti's conjecture also enables them to settle the above conjecture for the groups of type $B$. To be precise, Savage and Visontai \cite{Savage$s$} showed that the affine Eulerian polynomials $\widetilde{B}_{n}(x)$ is equal to $T_{n+1,n+1}(x)$ as defined by \eqref{refinement typed}, and derived the real-rootedness of $\widetilde{B}_{n}(x)$ from that of $T_{n+1,i}(x)$. This relation is obtained based on the following combinatorial interpretation of type $B$ affine descents:
\begin{align*}
\widetilde{\Des}_{B}\,\sigma = & \, \Des_{B}\sigma \cup\{n : \mathrm{if}\;\sigma_{n-1}+\sigma_{n}>0\}.
\end{align*}
There exists a similar combinatorial interpretation of type $D$ affine descents as follows:
\begin{align*}
\widetilde{\Des}_{D}\,\sigma = & \, \Des_{D}\sigma \cup\{n : \mathrm{if}\;\sigma_{n-1}+\sigma_{n}>0\}.
\end{align*}
However, the real-rootedness conjecture of $\widetilde{D}_{n}(x)$ still remains open.

In this paper, we completely solve Conjectures \ref{conj:q-type-D} and \ref{conj:affine-type-D}. Motivated by Savage and Visontai's  proof of the real-rootedness of ${D}_{n}(x)$, we are aimed at finding two families of polynomials, which not only constitute $D_n(x;q)$ and $\widetilde{D}_{n}(x)$ respectively, but also satisfy similar recurrence relations to those of $T_{n,i}(x)$. Based on these recurrences, we shall show that the interlacing property holds for these polynomials, from which we can derive the real-rootedness of $D_n(x;q)$ and $\widetilde{D}_{n}(x)$.

For the $q$-Eulerian polynomial $D_n(x;q)$, we find a refinement $T_{n,i}(x;q)$ as given in Section \ref{section-q-Eulerian}. But, due to the additional parameter $q$, we can not check the compatibility of $T_{n,i}(x;q)$ for small $n$ via a direct numerical analysis as done by Savage and Visontai for proving the compatibility of $T_{n,i}(x)$. To overcome this difficulty, we turn to check the interlacing property of $T_{n,i}(x;q)$ since all these polynomials have nonnegative coefficients. By the Hermite--Biehler theorem, the problem is further transformed to testing the stability of some polynomials associated with $T_{n,i}(x;q)$, which can be done by using the Routh--Hurwitz stability criterion.

For the affine Eulerian polynomial $\widetilde{D}_{n}(x)$, it is also hoped that
there exists a refinement satisfying a recurrence relation like \eqref{eq:T_n,i}.
Based on the combinatorial interpretation of type $D$ affine descents, Savage and Visontai \cite{Savage$s$} gave an expression of the affine Eulerian polynomial $\widetilde{D}_{n}(x)$ in terms of the ascent statistic over inversion sequences.
With this expression, it is natural to define a refinement of $\widetilde{D}_{n}(x)$ as done for ${D}_{n}(x)$.
Unfortunately, for such refined polynomials there does not exist a recurrence relation as that of $T_{n,i}(x)$. We construct a family of polynomials based on these refined polynomials, and show that these polynomials satisfy the same recurrence relation as $T_{n,i}(x)$. Then we prove the interlacing property of these polynomials, from which we derive the real-rootedness of $\widetilde{D}_{n}(x)$.

This paper is organized as follow.
In Section \ref{section-interlacing}, we give an overview of Savage and Visontai's general theorem on transformations preserving real-rootedness, which is the key tool for our proofs of Brenti's conjecture and Dilks, Petersen, and Stembridge's conjecture.
In Section \ref{section-q-Eulerian}, we give a refinement of $\widetilde{D}_n(x;q)$ and present a proof of Conjecture \ref{conj:q-type-D}.
In Section \ref{section-affine-Eulerian}, we define a new family of polynomials concerning $\widetilde{D}_{n}(x)$ and prove the real-rootedness of $\widetilde{D}_n(x)$, and hence confirm Conjecture \ref{conj:affine-type-D}. At the end of this paper, we give a new proof of an equality on $\widetilde{B}_{n}(x),\,\widetilde{D}_{n}(x)$
and $D_{n-1}(x)$ due to Dilks, Petersen, and Stembridge \cite{Dilks2009Affine}.

\section{Transformations preserving real-rootedness}\label{section-interlacing}

The aim of this section is to give an overview of
Savage and Visontai's theorem on transformations preserving real-rootedness.
This theorem not only plays an important role in their study of the real-rootedness of various Eulerian polynomials, but also is critical for
our proofs of Conjectures \ref{conj:q-type-D} and \ref{conj:affine-type-D}.
Due to its significance, we also present another proof of this result when
all the polynomials involved have only nonnegative coefficients.
Finally, we note a unified theorem of Fisk, of which the above two theorems can be treated as corollaries.

Savage and Visontai's theorem was originally stated in languages of compatible polynomials. Let us first introduce some related concepts. Suppose that
$f_{1}(x), \dots, f_{m}(x)$ are polynomials with real coefficients.
These polynomials are said to be compatible if, for any nonnegative numbers $c_{1},\dots,c_{m}$, the polynomial
$$c_1f_1(x)+c_2f_2(x)+\cdots+c_{m}f_{m}(x)$$
has only real zeros, and they are said to be pairwise compatible if, for all $1\le i<j\le m$, the polynomials $f_{i}(x)$ and $f_{j}(x)$ are compatible.
These concepts are defined by Chudnovsky and Seymour \cite{Chudnovsky2007roots} in their study of the real-rootedness of independence polynomials of claw-free graphs. The following remarkable lemma shows that how the two concepts are related.

\begin{lem} [{\cite[2.2]{Chudnovsky2007roots}}] \label{lem:pairwise compatible}
The polynomials $f_{1}(x),\dots,f_{m}(x)$ with positive leading coefficients are pairwise compatible if and  only if they are compatible.
\end{lem}

Savage and Visontai's theorem is concerned with the following transformation:
Given a sequence of polynomials $(f_{1}(x),\dots,f_{m}(x))$ with real coefficients, define another sequence of polynomials $(g_{1}(x),\dots,g_{m'}(x))$ by the equations
\begin{align} \label{comp-pre-tran}
g_{k}(x)=\sum_{\ell=1}^{t_{k}-1}xf_{\ell}(x)+\sum_{\ell=t_{k}}^{m}f_{\ell}(x),\quad\mathrm{for}\;1\le k\le m',
\end{align}
where $1\le t_{1}\le\dotso\le t_{m'}\le m+1$.
Savage and Visontai obtained the following useful result.

\begin{thm}[{\cite[Theorem 2.3]{Savage$s$}}]
\label{thm:compatibility}
Given a sequence of real polynomials $f_{1}(x),\dots,f_{m}(x)$
with positive leading coefficients, let $g_{1}(x),\dots,g_{m'}(x)$ be  defined as in \eqref{comp-pre-tran}. If, for all $1\le i<j\le m$,
\begin{itemize}
\item[(1)] $f_i(x)$ and $f_j(x)$ are compatible, and
\item[(2)] $x f_i(x)$ and $f_j(x)$ are compatible,
\end{itemize}
then, for all $1\le i<j\le m'$,
\begin{itemize}
\item[(1')] $g_i(x)$ and $g_j(x)$ are compatible, and
\item[(2')] $x g_i(x)$ and $g_j(x)$ are compatible.
\end{itemize}
\end{thm}

As pointed out by Savage and Visontai, the description of the above theorem can be simplified by using the notion of interlacing if the polynomials $f_{1}(x),\dots,f_{m}(x)$ have only nonnegative coefficients.
Given two real-rooted polynomials $f(x)$ and $g(x)$ with positive leading coefficients, let $\{u_i\}$ be the set of zeros of $f(x)$ and $\{v_j\}$ the set of zeros of $g(x)$.
We say that {$g(x)$ interlaces $f(x)$}, denoted $g(x)\sep f(x)$, if either
$\deg f(x)=\deg g(x)=n$ and
\begin{align}\label{alt-def}
v_n\le u_n\le v_{n-1}\le\cdots\le v_2\le u_2\le v_1\le u_1,
\end{align}
or $\deg f(x)=\deg g(x)+1=n$ and
\begin{align}
u_{n}\le v_{n-1}\le\cdots\le v_{2}\le u_{2}\le v_{1}\le u_{1}.\label{int-def}
\end{align}
If all inequalities in \eqref{alt-def} or \eqref{int-def} are strict,
then we say that {$g(x)$ strictly interlaces $f(x)$}, denoted $g(x)\seps f(x)$. Parallel to the concept of pairwise compatibility, we say that a sequence of real polynomials $(f_{1}(x),\dots,f_{m}(x))$
with positive leading coefficients is mutually interlacing if $f_{i}(x) \sep f_{j}(x)$ for all $1\le i<j\le m$. As far as we know, this definition was first introduced by Fisk \cite{FiskPolynomials}.
It should be mentioned that the notations of interlacing and mutual interlacing adopted in this paper are a little different from Fisk's.
Interlacing of two polynomials is closely related to compatibility
in the sense of the following, due to Wagner \cite{Wagner2000Zeros}.
\begin{thm}[{\cite[Lemma 3.4]{Wagner2000Zeros}}]
\label{thm-wagner}
Suppose that $f(x)$ and $g(x)$ are two polynomials with nonnegative coefficients. Then the following statements are equivalent:
\begin{itemize}
\item[(1)]  $f(x)$ interlaces $g(x)$;

\item[(2)]  $f(x)$ and $g(x)$ are compatible, and  $xf(x)$ and $g(x)$ are compatible.
\end{itemize}
\end{thm}

With the above theorem, we now give an alternative description of Theorem \ref{thm:compatibility} when all the polynomials involved have only nonnegative coefficients.

\begin{thm}[{\cite[Theorem 2.4]{Savage$s$}}]
\label{thm:interlacing}
Given a sequence of polynomials $(f_{1}(x),\dots,f_{m}(x))$
with nonnegative coefficients, let $g_{1}(x),\dots,g_{m'}(x)$ be polynomials defined as in \eqref{comp-pre-tran}. If $(f_{1}(x),\dots,f_{m}(x))$ is mutually interlacing, then so is $(g_{1}(x),\dots,g_{m'}(x))$.
\end{thm}

In the following we shall give a new proof of Theorem \ref{thm:interlacing}
with the notion of interlacing. It is our feeling that, for polynomials with nonnegative coefficients, it is more convenient to work with interlacing than with compatibility.
Before giving our proof, let us first note the following useful lemma, and the proof is omitted here.

\begin{lem}
[{\cite[Lemma 2.3]{Branden2006linear}, \cite[Proposition 3.5]{Wagner1992Total}}] \label{lem:roots-interlacing}
Let $g(x)$ and $\{f_{i}(x)\}_{i=1}^{n}$ be real-rooted polynomials with
positive leading coefficients, and let $F(x)=f_{1}(x)+f_{2}(x)+\cdots+f_{n}(x)$. Then
\begin{itemize}
\item[(1)] if $f_{i}(x) \sep g(x)$ for each $1\le i\le n$,
then $F(x)$ is real-rooted with $F(x)\sep g(x)$;

\item[(2)] if $g(x) \sep f_{i}(x)$ for each $1\le i\le n$,
then $F(x)$ is real-rooted with
$g(x) \sep F(x)$.
\end{itemize}
\end{lem}

Using the above lemma, we immediately have the following result and the proof is also omitted here. We would like to point out that it can also be proved by using compatibility.

\begin{lem}[{\cite[Corollary 3.6]{FiskPolynomials}}]\label{mutual-inter-cor}
Let $(f_1(x),f_2(x), \ldots, f_m(x))$ be a sequence of polynomials with
positive leading coefficients. If it is mutually interlacing, then, for any $1\le i< j \le m$, the polynomial $f_i(x)+\cdots+f_{j}(x)$ has only real zeros, and moreover,
$$f_i(x)\sep f_i(x)+\cdots+f_{j}(x) \sep f_{j}(x).$$
\end{lem}

We proceed to prove Theorem \ref{thm:interlacing}.

\noindent\textit{Proof of Theorem \ref{thm:interlacing}.}
We first show that for any $1\le i\le m$ the polynomial $g_i(x)$ has only real zeros. Note that
\begin{align*}
g_{i}(x)  = & x\sum_{\alpha=1}^{t_{i}-1}f_{\alpha}(x)+\sum_{\beta=t_{i}}^{m}f_{\beta}(x).
\end{align*}
By the mutual interlacing of $(f_{1}(x),\dots,f_{m}(x))$, we get
$$f_{\alpha}(x)\sep f_{\beta}(x)$$
for any $1\le \alpha\le t_i-1$ and $t_i\le \beta \le m$.
Since both $f_{\alpha}(x)$ and $f_{\beta}(x)$ are real-rooted polynomials with nonnegative coefficients, we then have
$$f_{\beta}(x)\sep xf_{\alpha}(x)$$
for any $1\le \alpha \le t_i-1$ and $t_i\le \beta \le m$.
From (1) of Lemma \ref{lem:roots-interlacing} we deduce that
$$\sum_{\beta=t_i}^m f_{\beta}(x)\sep xf_{\alpha}(x)$$
for any $1\le \alpha \le t_i-1$.
Further, by (2) of Lemma \ref{lem:roots-interlacing}, we obtain
$$\sum_{\beta=t_i}^m f_{\beta}(x)\sep x\sum_{\alpha=1}^{t_i-1}f_{\alpha}(x).$$
Again, by (2) of Lemma \ref{lem:roots-interlacing}, it follows that
$$\sum_{\beta=t_i}^m f_{\beta}(x)\sep x\sum_{\alpha=1}^{t_i-1}f_{\alpha}(x)+\sum_{\beta=t_i}^m f_{\beta}(x)=g_i(x),$$
which implies the real-rootedness of $g_i(x)$.

Now we can prove the mutual interlacing of $(g_1(x), g_2(x), \ldots, g_{m'}(x))$. By definition, it suffices to show that $g_i(x)$ interlaces $g_j(x)$ for all $1\le i<j\le m'$. Without loss of generality, we may assume that $g_i(x)\neq g_j(x)$, namely, $t_i< t_j$. It is obvious that
\begin{align*}
g_{j}(x)  = & \
x\sum_{\alpha=1}^{t_{j}-1}f_{\alpha}(x)+\sum_{\beta=t_{j}}^{m}f_{\beta}(x)\\
 = & \ x\sum_{\alpha=1}^{t_{i}-1}f_{\alpha}(x)+\sum_{\beta=t_{i}}^{m}f_{\beta}(x)+(x-1)\sum_{\gamma=t_{i}}^{t_{j}-1}f_{\gamma}(x)\\
  = &\  g_{i}(x)+(x-1)\sum_{\gamma=t_{i}}^{t_{j}-1}f_{\gamma}(x).
\end{align*}
By Lemma \ref{mutual-inter-cor}, the polynomial $\sum_{\gamma=t_{i}}^{t_{j}-1}f_{\gamma}(x)$ has only real zeros.
To prove $g_{i}(x)\sep g_{j}(x)$, by (2) of Lemma \ref{lem:roots-interlacing}, it suffices to show that
\begin{align*}
g_i(x)\sep (x-1)\sum_{\gamma=t_{i}}^{t_{j}-1}f_{\gamma}(x),
\end{align*}
which is equivalent to
\begin{align*}
\sum_{\gamma=t_{i}}^{t_{j}-1}f_{\gamma}(x) \sep g_i(x)=x\sum_{\alpha=1}^{t_{i}-1}f_{\alpha}(x)+\sum_{\beta=t_{i}}^{m}f_{\beta}(x).
\end{align*}
Again, by (2) of Lemma \ref{lem:roots-interlacing}, we only need to show that
\begin{align}\label{equ-partab}
\sum_{\gamma=t_{i}}^{t_{j}-1}f_{\gamma}(x) \sep x\sum_{\alpha=1}^{t_{i}-1}f_{\alpha}(x) \quad \mbox{and} \quad \sum_{\gamma=t_{i}}^{t_{j}-1}f_{\gamma}(x)\sep \sum_{\beta=t_{i}}^{m}f_{\beta}(x).
\end{align}
By Lemma \ref{mutual-inter-cor}, for any $t_{i}\le\gamma<t_{j}$, we have
$$\sum_{\alpha=1}^{t_{i}-1}f_{\alpha}(x)\ \sep\  f_{\gamma}(x).$$
Then, by (2) of Lemma \ref{lem:roots-interlacing}, it follows that
$$\sum_{\alpha=1}^{t_{i}-1}f_{\alpha}(x)\ \sep\ \sum_{\gamma=t_{i}}^{t_{j}-1} f_{\gamma}(x),$$
which is equivalent to the first part of \eqref{equ-partab}.

The second part of  \eqref{equ-partab}  can be proved along the same line.
By Lemma \ref{mutual-inter-cor}, for any $t_j\le \beta\le m$,
we have
$$\sum_{\gamma=t_{i}}^{t_{j}-1}f_{\gamma}(x)\sep f_{\beta}(x).$$
Then, by (2) of Lemma \ref{lem:roots-interlacing},
$$\sum_{\gamma=t_{i}}^{t_{j}-1}f_{\gamma}(x)\sep \sum_{\beta=t_j}^{m} f_{\beta}(x),$$
and hence
$$\sum_{\gamma=t_{i}}^{t_{j}-1}f_{\gamma}(x)\sep \sum_{\gamma=t_{i}}^{t_{j}-1}f_{\gamma}(x)+\sum_{\beta=t_j}^{m} f_{\beta}(x)=\sum_{\gamma=t_{i}}^{m}f_{\gamma}(x),$$
as desired.
This completes the proof.
\qed

Theorem \ref{thm:interlacing} is efficient for proving the mutual interlacing of the refined Eulerian polynomials satisfying recurrence relations as in \eqref{eq:T_n,i}. However, it is not enough for our proof of the real-rootedness of the affine Eulerian polynomials of type $D$. We also need the following result due to Haglund, Ono, and Wagner.

\begin{thm}[{\cite[Lemma~8]{Haglund1999Theorems}}]\label{Haglund1999}
Let $f_1(x), \dots, f_m(x)$ be real-rooted polynomials with nonnegative coefficients, and let  $a_1, \dots, a_m \ge 0$  and $b_1, \dots, b_m \ge 0$ be such that
$a_{i}b_{i+1} \ge b_{i}a_{i+1}$ for all $1 \le i \le m-1.$ If the sequence $(f_1(x), \dots, f_m(x))$ is mutually interlacing, then
$\sum_{i=1}^{m}a_{i}f_{i}(x)$ interlaces $\sum_{i=1}^{m}b_{i}f_{i}(x)$.
\end{thm}

Savage and Visontai \cite{Savage$s$} mentioned that Theorem \ref{thm:compatibility} together with Theorem \ref{thm-wagner} implies Theorem \ref{Haglund1999}. Note that Theorem \ref{thm:interlacing} is also implied by Theorems \ref{thm:compatibility} and \ref{thm-wagner}.
In the following we shall give a unified interpretation of Theorems
\ref{thm:interlacing} and Theorem \ref{Haglund1999} via Fisk's theory on matrices preserving mutual interlacing \cite{FiskPolynomials}.
Following Fisk, we say that a matrix $M=(m_{i,j})$ is an $NX$ matrix if its entries are either nonnegative constants or positive multiples of $x$.
Fisk gave the following criterion to determine whether an $NX$ matrix preserves the mutually interlacing property.

\begin{thm}[{\cite[Propostion 3.72]{FiskPolynomials}}]\label{thm:fisk}
An $NX$ matrix $M=(m_{i,j})$ preserves mutually interlacing sequences of polynomials with nonnegative coefficients if the following conditions are satisfied:
\begin{itemize}
\item[(1)] All entries that lie in the southwest of a multiple of $x$ are multiples of $x$.

\item[(2)] For any two by two submatrix of $M$ having the form
$
\begin{pmatrix}
a & b\\
c & d
\end{pmatrix} \mbox{or}
\begin{pmatrix}
ax & bx\\
cx & dx
\end{pmatrix},
$
we have $ad-bc\geq 0$.

\item[(3)] For any two by two submatrix of $M$ having the form
$
\begin{pmatrix}
a & b\\
cx & dx
\end{pmatrix}\mbox{or}
\begin{pmatrix}
ax & b\\
cx & d
\end{pmatrix},
$
we have $ad-bc\leq 0$.
\end{itemize}
\end{thm}

Based on the above theorem, we could have a new look at Theorems
\ref{thm:interlacing} and \ref{Haglund1999}.
The latter could be considered as a special case of Theorem \ref{thm:fisk} with the corresponding $NX$ matrix given by
$$\begin{pmatrix}
a_1 & a_2 & \dots & a_n\\
b_1 & b_2 & \dots & b_n
\end{pmatrix}.
$$
For the former, the corresponding $NX$ matrix is a $\{1,x\}$-matrix with its $i$-th row given by
$$(\underset{(t_i-1)'s}{\underbrace{x,x,\ldots,x}}, \underset{(m-t_i+1)'s}{\underbrace{1,1,\ldots,1}}).$$
It is easy to verify that both matrices satisfy the conditions in Theorem \ref{thm:fisk}.

\section{The \texorpdfstring{$q$}{Lg}-Eulerian polynomials of type \texorpdfstring{$D$}{Lg}}
\label{section-q-Eulerian}

In this section we aim to prove the real-rootedness of the $q$-Eulerian polynomials $D_n(x;q)$ for any positive $q$.
Brenti \cite{Brenti1994$q$} noted that the type $D$ statistics $\nege_D$ and ${\des}_{D}$ can be extended to all signed permutations, and proved that \begin{align*}
(1+q)D_n(x;q) = \sum_{\sigma\in{{B}_n}}q^{\nege \sigma}x^{{\des}_{D}\sigma}.
\end{align*}
Let
\begin{align*}
T_n(x;q)= \sum_{\sigma\in{{B}_n}}q^{\nege \sigma}x^{{\des}_{D}\sigma}.
\end{align*}
It is obvious that $T_n(x;q)$ has only real zeros if and only if
$D_n(x;q)$ has only real zeros. In the following we shall focus on the real-rootedness of $T_n(x;q)$.

To prove that $T_n(x;q)$ has only real zeros for positive $q$,
let us first give a proper refinement of $T_n(x;q)$.
To this end, we need to interpret $T_n(x;q)$ as the generating functions of certain statistics over inversion sequences. This could be easily done by using a map $\psi:{{B}_n}\rightarrow\I_{n}^{(2,4,\dots)}$ established by
Savage and Visontai \cite{Savage$s$}. Precisely, a signed permutation $\sigma=(\sigma_{1},\ldots,\sigma_{n})\in{{B}_n}$ under $\psi$  is mapped to an inversion sequence $(e_{1},\ldots,e_{n})\in \I_{n}^{(2,4,\dots)}$ given by, for $1\le i \le n $,
\begin{align*}
e_{i}=\begin{cases}
t_{i} & \mbox{if} \quad \sigma_{i}>0\,,\\
2i-t_{i}-1 & \mbox{if} \quad \sigma_{i}<0,
\end{cases}
\end{align*}
where $t_{i}\ =\ \left|\{j\in[i-1] : |\sigma_{j}|>|\sigma_{i}|\}\right|.$
The map $\psi$ satisfies the following properties.

\begin{lem}[{\cite[Theorem 3.12]{Savage$s$}}]
\label{bijection}
The map $\psi:{{B}_n}\rightarrow\I_{n}^{(2,4,\dots)}$ is a bijection satisfying the following properties:
\begin{itemize}
\item[(1)] $\sigma_{i}<0$ if and only if $e_{i}\ge i$, for any $1\le i \le n$;
\item[(2)] $\sigma_{1}+\sigma_{2}<0$ if and only if $e_{1}+\tfrac{e_{2}}{2}\geq \tfrac{3}{2}$;
\item[(3)] $\sigma_{i}>\sigma_{i+1}$ if and only if $\tfrac{e_{i}}{i}<\tfrac{e_{i+1}}{i+1}$
for $1\le i\le n-1$;
\item[(4)] $\sigma_{n-1}+\sigma_{n}>0$ if and only if $\tfrac{e_{n-1}}{n-1}+\tfrac{e_{n}}{n}<\tfrac{n-1}{n}$.
\end{itemize}
\end{lem}

For an inversion sequence $\e=(e_{1},\dots,e_{n})\in\I_{n}^{(2,4,\dots)}$, let
$$\exc \e= \sum_{i=1}^n\chi \left(e_i\ge i\right).$$
By (1) of Lemma \ref{bijection}, we see that $\exc \e=\nege \psi^{-1}(\e)$.
By (2) and (3) of Lemma \ref{bijection}, we have
$\asc_D \e=\des_D \psi^{-1}(\e)$. The following result is immediate.

\begin{lem}\label{q-trans} For any $n\geq 2$, we have
\begin{align*}
 T_n(x;q) = \sum_{\e\in\I_{n}^{(2,4,\dots)}}q^{\exc \e}x^{\asc_{D}\e}.
\end{align*}
\end{lem}

Now we can give a refinement of $T_n(x;q)$.
Let
\begin{align*}
T_{n,i}(x;q)\ =\  \sum_{\e\in\I_{n}^{(2,4,6,\dotsc)}}\chi(e_{n}=i)\, q^{\exc \e} x^{\asc_{D}\e}.
\end{align*}
It is clear that
$$T_n(x;q)=\sum_{i=0}^{2n-1}T_{n,i}(x;q), \qquad T_{n,i}(x;1)=T_{n,i}(x).$$
These polynomials $T_{n,i}(x;q)$ satisfy the following recurrence relation.
\begin{lem}
\label{q-T-recurrence}
For $n\geq3$ and $0\le i\le 2n-1,$ we have
\begin{align}\label{eq:Tq_n,i}
T_{n,i}(x;q)\ =\ q^{\chi(i\ge n)}\Big( x\sum_{j=0}^{\left\lceil \tfrac{n-1}{n}i\right\rceil -1}T_{n-1,j}(x;q)+\sum_{j=\left\lceil \tfrac{n-1}{n}i\right\rceil }^{2n-3}T_{n-1,j}(x;q)\Big)\,,
\end{align}
with the initial conditions that $T_{2,0}(x;q)=1+q$, $T_{2,1}(x;q)=(1+q)x$, $T_{2,2}(x;q)=(q+q^2)x$, and $T_{2,3}(x;q)=(q+q^2)x^{2}$.
In particular, $T_{n,0}(x;q)\ = \ T_{n-1}(x;q)$.
\end{lem}
\begin{proof}
For the initial values, it is easy to verify.
For $\e=(e_1, \ldots, e_{n-1}, i) \in \I_{n}^{(2,4,6,\dotsc)}$, it is clear that
$$\exc \e= \exc (e_1, \ldots, e_{n-1}) + \chi(i\ge n)\,.$$
Moreover, we have that $n-1 \in \Asc _D \e$ if and only if
$$\frac{e_{n-1}}{n-1} < \frac{i}{n}\,,$$
that is, $e_{n-1} < \tfrac{n-1}{n}i.$
Taking $i=0$ in \eqref{eq:Tq_n,i}, it is readily to see that
$$T_{n,0}(x;q) = \sum_{j=0}^{2n-3}T_{n-1,j}(x;q) =  T_{n-1}(x;q).$$
This completes the proof.
\end{proof}

To show the real-rootedness of  $T_n(x;q)$, we further need to prove that the sequence $(T_{n,i}(x;q))_{i=0}^{2n-1}$ is mutually interlacing.
With the above recurrence relation, it is desirable to give an induction proof as done by Savage and Vistonai for the polynomials $T_{n,i}(x)$.
For the basis step of the induction, Savage and Vistonai showed that the sequence $(T_{4,i}(x))_{i=0}^{7}$ is mutually interlacing
by numerical calculations. It is hoped that for any positive $q$ the sequence $(T_{4,i}(x;q))_{i=0}^{7}$ is also mutually interlacing. In fact, this is true, as shown in Lemma \ref{lem:mutual}. However, due to the additional parameter $q$, we can not directly follow the way of Savage and Vistonai to verify the interlacing. To fix this, we shall use the Hermite--Biehler theorem and the Routh--Hurwitz criterion for stability of complex polynomials, as illustrated below.

The Hermite--Biehler theorem presents necessary and sufficient conditions for the stability of a polynomial in terms of certain interlacing conditions. Recall that a complex polynomial $p(z)$ is said to be {Hurwitz stable} (respectively, {weakly Hurwitz stable}) if $p(z)\neq0$ whenever $\Re(z)\geq0$ (respectively, $\Re(z)>0$), where $\Re(z)$ denotes the real part of $z$.
Suppose that
$p(z)=\sum_{k=0}^{n}a_{n-k}z^{k}.$
Let
\begin{align}\label{eq-even-odd}
p^{E}(z)=\sum_{k=0}^{\lfloor n/2\rfloor}a_{n-2k}z^{k}\quad\mbox{ and }\quad p^{O}(z)=\sum_{k=0}^{\lfloor(n-1)/2\rfloor}a_{n-1-2k}z^{k}.
\end{align}
The Hermite--Biehler theorem could be stated as follows, which establishes a connection between the interlacing property between $p^{E}(z)$ and $p^{O}(z)$ and the stability of $p(z)$.

\begin{thm}[{\cite[Theorem 4.1]{Branden2011Iterated}}]\label{Hermite--Biehler}
Let $p(z)$ be a polynomial with real coefficients, and let $p^{E}(z)$ and $p^{O}(z)$ be defined as in \eqref{eq-even-odd}. Suppose that $p^{E}(z)p^{O}(z)\not\equiv0$. Then $p(z)$
is Hurwitz stable (resp. weakly Hurwitz stable) if and only if $p^{E}(z)$ and $p^{O}(z)$ have
only negative (resp. nonpositive)
zeros, and moreover $p^{O}(z) \seps p^{E}(z)$ (resp. $p^{O}(z) \sep p^{E}(z)$).
\end{thm}

Therefore, to prove the mutual interlacing of the sequence $(T_{4,i}(x;q))_{i=0}^{7}$ for positive $q$, it suffices to show that the polynomial
$$zT_{4,i}(z^2;q)+T_{4,j}(z^2;q)$$
is weakly Hurwitz stable for any $0\leq i<j\leq 7$. A useful criterion for determining stability was given by Hurwitz \cite{Hurwitz1895Ueber}, which we shall explain below.
Given a polynomial $p(z)=\sum_{k=0}^{n}a_{n-k}z^{k}$, for any $1\leq k\leq n$ let
\begin{align*}
  \Delta_k(p)= \det \left(
  \begin{array}{ccccc}
   a_1 & a_3 & a_5 & \dots & a_{2k-1}\\
   a_0 & a_2 & a_4 & \dots & a_{2k-2}\\
    0  & a_1 & a_3 & \dots & a_{2k-3}\\
    0  & a_0 & a_2 & \dots & a_{2k-r}\\
  \dots&\dots&\dots& \dots & \dots\\
    0  &  0  &  0  & \dots & a_{k}\\
  \end{array}\right)_{k\times k}.
\end{align*}
These determinants are known as the Hurwitz determinants of $p(z)$.
Hurwitz showed that the stability of $p(z)$ is uniquely determined by the signs of $\Delta_k(p)$.

\begin{thm}[{\cite{Hurwitz1895Ueber}}]\label{Hurwitz criterion}
Suppose that $p(z)=\sum_{k=0}^{n}a_{n-k}z^{k}$ is a real polynomial with $a_0>0$. Then
$p(z)$ is Hurwitz stable if and only if the corresponding Hurwitz determinants $\Delta_k(p)>0$ for any $1\leq k\leq n$.
\end{thm}

The above result is
usually called the Routh-Hurwitz stability criterion since it is
equivalent to the Routh test, for more historical background see \cite[pp. 393]{Rahman2002Analytic}. With this criterion, we are able to prove the mutually interlacing properpty of $(T_{4,i}(x;q))_{i=0}^{7}$.

\begin{lem}\label{lem:mutual}
For any positive $q$, the sequence $(T_{4,i}(x;q))_{i=0}^{7}$
is mutually interlacing.
\end{lem}

\begin{proof}
By Lemma \ref{q-T-recurrence}, it is easy to compute that
\begin{align*}
\begin{tabular}{ll}
\ensuremath{T_{4,0}(x;q)=}  \ensuremath{ (q+1) \left(q^2 x^3+(4 q^2+6 q+1) x^2+(q^2+6 q+4) x+1\right)},      &     \\[5pt]
\ensuremath{T_{4,1}(x;q)=}   \ensuremath{(q+1) \left((q^2+q) x^3+(4 q^2+6 q+2) x^2+(q^2+5 q+4) x\right)},        &    \\[5pt]
\ensuremath{T_{4,2}(x;q)=}   \ensuremath{(q+1) \left((q^2+2 q) x^3+(4 q^2+6 q+4) x^2+(q^2+4 q+2) x\right)},        &    \\[5pt]
\ensuremath{T_{4,3}(x;q)=}   \ensuremath{(q+1) \left((q^2+3 q+1) x^3+(4 q^2+6 q+4) x^2+(q^2+3 q+1) x\right)}, &    \\[5pt]
\ensuremath{T_{4,4}(x;q)=}   \ensuremath{(q+1) \left(q (q^2+3 q+1) x^3+q (4 q^2+6 q+4) x^2+q (q^2+3 q+1) x\right)}, &    \\[5pt]
\ensuremath{T_{4,5}(x;q)=}   \ensuremath{(q+1) \left(q (2 q^2+4 q+1) x^3+q (4 q^2+6 q+4) x^2+q (2 q+1) x\right)},  &    \\[5pt]
\ensuremath{T_{4,6}(x;q)=}   \ensuremath{(q+1) \left(q (4 q^2+5 q+1) x^3+q (2 q^2+6 q+4) x^2+q (q+1) x \right)},  &    \\[5pt]
\ensuremath{T_{4,7}(x;q)=}   \ensuremath{(q+1) \left(q^3 x^4+q (4 q^2+6 q+1) x^3+q (q^2+6 q+4) x^2+q x\right)}. &
\end{tabular}
\end{align*}
Note that
$$T_{4,4}(x;q)=qT_{4,3}(x;q) \quad \mbox{and} \quad T_{4,7}(x;q)=qxT_{4,0}(x;q).$$
Thus, it suffices to show that
$T_{4,i}(x;q)\sep T_{4,j}(x;q)$
for any $q>0$  and $i<j$ with $i, j \in \{0,1,2,3,5,6\}$.
For the case of $q=1$, the interlacing property has already been obtained by
Savage and Visontai \cite{Savage$s$}. In the following we shall assume that $q\neq 1$.

By Theorem \ref{Hermite--Biehler}, we only need to prove that
$$T_{4,j}(z^2;q) + z\,T_{4,i}(z^2;q)$$ is weakly Hurwitz stable for any positive $q\neq 1$ and  $i<j$ with $i, j \in \{0,1,2,3,5,6\}$. In fact, all these polynomials are Hurwitz stable up to a power of $z$. Let
\begin{align*}
 C_{i,j}(z)=\frac{T_{4,j}(z^2;q) + z\,T_{4,i}(z^2;q)}{z^{m_{i,j}} (q+1)},
\end{align*}
where $m_{i,j}$ is the largest nonnegative integer $k$ such that
$$z^k\,|\,(T_{4,j}(z^2;q) + z T_{4,i}(z^2;q)).$$

We proceed to show that $C_{i,j}(z)$ is Hurwitz stable for any $i<j$ with $i,j \in \{0,1,2,3,5,6\}$. By Theorem \ref{Hurwitz criterion}, we only need to show that all the Hurwitz determinants of $C_{i,j}(z)$ are positive.
It is easy to compute these Hurwitz determinants with the aid of a computer.
We would like to mention that most of these determinants are nonzero polynomials of $q$ with only nonnegative coefficients except for $(i,j)\in \{(0,1), (0,6),(1,6)\}$. Therefore, if $i<j$ and $(i,j)\not\in \{(0,1), (0,6),(1,6)\}$, then the corresponding Hurwitz determinants of $C_{i,j}(z)$ must be positive for any positive $q$, hence establishing its Hurwitz stability. In the following, we shall separately check the sign of the Hurwitz determinants of $C_{i,j}(z)$ for $(i,j)\in \{(0,1), (0,6),(1,6)\}$.

For $(i,j)= (0,1)$, the testing polynomial is
\begin{align*}
 C_{0,1}(z)  = & \,q^2 z^6+(q^2+q) z^5+(4 q^2+6 q+1) z^4+(4 q^2+6 q+2) z^3\\
               & +(q^2+6 q+4) z^2+(q^2+5 q+4) z+1,
\end{align*}
and the corresponding Hurwitz determinants are
$$
\begin{array}{ll}
\Delta_1 = q (q+1), & \Delta_2 = q (4 q^2+5 q+1),\\[5pt]
\Delta_3  = 2 q (q+1)^2 (7 q^2+4 q+1), & \Delta_4 = 4 q (q+1)^2 (3 q^3+q^2+q+1),\\[5pt]
\Delta_5  = 12 q (q+1)^3 (q^2-1)^2, &
\Delta_6 = 12 q (q+1)^3 (q^2-1)^2.
\end{array}
$$
It is obvious that they are all positive for any positive $q\neq 1$ .
This means that $C_{0,1}(z)$ is Hurwitz stable for any positive $q\neq 1$.

For  $(i,j)=(0,6)$, the testing polynomial is
  \begin{align*}
 C_{0,6}(z)  = & \,q^2 z^6+(4 q^3+5 q^2+q) z^5+(4 q^2+6 q+1) z^4+(2 q^3+6 q^2+4 q) z^3\\
               & +(q^2+6 q+4) z^2+(q^2+q) z+1,
\end{align*}
and the corresponding Hurwitz determinants are
$$
\begin{array}{ll}
\Delta_1 = q (4 q^2+5 q+1),  \qquad \quad\ \Delta_2 = q (14 q^4+38 q^3+34 q^2+11 q+1), \\[5pt]
\Delta_3 = 4 q^3 (q+1)^2 (3 q^3+q^2+q+1), \\[5pt]
\Delta_4 = 2 q^3 (q+1)^2 (6 q^5+12 q^4-12 q^3-11 q^2+10 q+7), \\[5pt]
\Delta_5 = 12 q^5 (q+1)^3 (q^2-1)^2,\quad
\Delta_6 = 12 q^5 (q+1)^3 (q^2-1)^2.
\end{array}
$$
For any positive $q\neq 1$, it is clear that all Hurwitz determinants is positive except for $\Delta_4$. For any $q>0$, we can verify that
$\Delta_4>0$ by numerical analysis with the aid of a computer. Thus, $C_{0,6}(z)$ is Hurwitz stable for any positive $q\neq 1$.

For $(i,j)=(1,6)$, the testing polynomial is
 \begin{align*}
 C_{1,6}(z)  = & \,(q+1) q z^5+(q+1) (4 q^2+q) z^4+(q+1) (4 q+2) z^3\\
               & +(q+1) (2 q^2+4 q) z^2+(q+1) (q+4) z+(q+1) q,
\end{align*}
and the corresponding Hurwitz determinants
$$
\begin{array}{ll}
\Delta_1 = q (4 q^2+5 q+1), & \Delta_2 = 2 q (q+1)^2 (7 q^2+4 q+1), \\[5pt]
\Delta_3 = 4 q^2 (q+1)^3 (3 q^3+q^2+q+1), &
\Delta_4 = 12 q^2 (q+1)^4 (q^2-1)^2, \\[5pt]
\Delta_5 = 12 q^3 (q+1)^5 (q^2-1)^2 &
\end{array}
$$
are positive for any positive $q\neq 1$. Thus, $C_{1,6}(z)$ is Hurwitz stable for any positive $q\neq 1$.

Combining the above cases, we get the stability of $C_{i,j}(z)$ for any
$i<j$ and $i,j \in \{0,1,2,3,5,6\}$, which implies the mutual interlacing of
the sequence $(T_{4,i}(x;q))_{i=0}^{7}$. This completes the proof.
\end{proof}

Now we can prove the mutual interlacing of $(T_{n,i}(x;q))_{i=0}^{2n-1}$ for general $n$.

\begin{prop}\label{q-T-interlacing}
For $n\ge 4$ and any positive $q$, the sequence of polynomials $(T_{n,i}(x;q))_{i=0}^{2n-1}$ is mutually interlacing.
\end{prop}

\begin{proof}
We use induction on $n$. When $n = 4$, the statement is true by Lemma \ref{lem:mutual}.
Note that $(T_{n,i}(x;q))_{i=0}^{2n-1}$ is mutually interlacing if and only if $(q^{-\chi(i\ge n)} T_{n,i}(x;q))_{i=0}^{2n-1}$ is mutually interlacing.
So, it suffices to prove that the sequence of polynomials
$$\left(q^{-\chi(i\ge n)}T_{n,i}(x;q)\right)_{i=0}^{2n-1}$$
is mutually interlacing.
By the recurrence \eqref{eq:Tq_n,i}, it is easy to see that the polynomials $q^{-\chi(i\ge n)} T_{n,i}(x;q)$ satisfy the conditions required in Theorem \ref{thm:interlacing}. By induction, the desired result immediately follows. This completes the proof.
\end{proof}

The main result of this section is as follows, which gives a positive answer to Conjecture \ref{conj:q-type-D}.

\begin{thm}
\label{q-invseqD} For $n\ge2$ and any positive $q$, the polynomial $D_n(x;q)$ has only real zeros.
\end{thm}

\begin{proof}
As mentioned at the beginning of this section, the real-rootedness of $D_{n}(x;q)$ is equivalent to that of $T_{n}(x;q)$. We shall prove that
for $n\ge2$ and $q>0$ the polynomial $T_n(x;q)$ has only real zeros.
This is true for $n=2$, since
$T_2(x;q)=(1+q)(1+x)(1+qx)$.
By Proposition \ref{q-T-interlacing}, we know that $(T_{n,i}(x;q))_{i=0}^{2n-1}$ is mutually interlacing for $n\geq 4$ and $q>0$. This also implies that $T_{n,0}(x;q)$ is real-rooted for any $n\geq 4$ and positive $q$.  Then by the equality $T_{n,0}(x;q)=T_{n-1}(x;q)$ in Lemma \ref{q-T-recurrence}, we obtain the desired result for $n\ge 3$.
This completes the proof.
\end{proof}

\section{The affine Eulerian polynomials of type \texorpdfstring{$D$}{Lg}}\label{section-affine-Eulerian}

The aim of this section is to prove the real-rootedness of the affine Eulerian polynomials $\widetilde{D}_n(x)$, which was conjectured by Dilks, Petersen, and Stembridge. The key ingredient of this section is to give a family of polynomial associated with $\widetilde{D}_n(x)$, with which we can give a proof of Dilks, Petersen, and Stembridge' conjecture as we have done for the $q$-Eulerian polynomials.

Our proof is based on an expression of $\widetilde{D}_n(x)$ in terms of the ascent statistic over inversion sequences, which was given by Savage and Visontai \cite{Savage$s$}. By using (2), (3) and (4) of Lemma \ref{bijection}, they showed that
\begin{align*}
2 \widetilde{D}_n(x)\ =\
\sum_{\e \in \I_{n}^{(2,4,\dots)} }x^{\widetilde{\asc}_{D}e},
\end{align*}
where the type $D$ affine ascent statistic for $\e\in \I_{n}^{(2,4,6,\dots)}$
is defined as
\begin{align}\label{def:affine-asc-D}
\widetilde{\asc}_{D}\e \ =\  {\asc}_{D}\e +\chi \big(\frac{e_{n-1}}{n-1}+\frac{e_n}{n}<\frac{2n-1}{n}\big).
\end{align}
Let $\widetilde{T}_n(x)= 2 \widetilde{D}_n(x)$.
It is natural to consider a refinement of the polynomial $\widetilde{T}_n(x)$, which is similar to $T_{n,i}(x)$:
\begin{align*}
\widetilde{T}_{n,i}(x)\ =\ \sum_{\e\in\I_{n}^{(2,4,\dots)}}\chi(e_{n}=i)\, x^{\widetilde{\asc}_{D}\e}\,.
\end{align*}
It is clear that
$$\widetilde{T}_n(x) = \sum_{i=0}^{2n-1}\widetilde{T}_{n,i}(x).$$
Unfortunately, this refinement is not an ideal one, because the polynomials
$\widetilde{T}_{n,i}(x)$ do not have a recurrence relation as that satisfied by $T_{n,i}(x)$.
However, $\widetilde{T}_{n,i}(x)$ still play a role in our proof, since they can be expressed in terms of $T_{n-1,i}(x)$ as follows.

\begin{lem}
\label{lem:S-recurrence}For $n\geq3$ and $0\le i \le n-1,$ we have
\begin{align}\label{eq:widetilde-T-recurrence}
\widetilde{T}_{n,i}(x)\ =\ x^{2}\,\sum_{j=0}^{i-1}T_{n-1,j}(x)+x\sum_{j=i}^{2n-i-3}T_{n-1,j}(x)+\sum_{j=2n-i-2}^{2n-3}T_{n-1,j}(x)\,,
\end{align}
and
\begin{align}\label{eq:widetilde-T-dual}
\widetilde{T}_{n,2n-i-1}(x) = \widetilde{T}_{n,i}(x)\,.
\end{align}
\end{lem}

\begin{proof}
By \eqref{def:asc-D} and \eqref{def:affine-asc-D}, for $\e=(e_{1},\ldots,e_{n})\in \I_{n}^{(2,4,\dots)}$ we have

\begin{align*}
\widetilde{\asc}_{D}(\e)
= \  & {\asc}_{D}\,(e_{1},\ldots,e_{n-1} )
+ \chi \big(\frac{e_{n-1}}{n-1} < \frac{e_n}{n} \big)
+ \chi \big(\frac{e_{n-1}}{n-1}+\frac{e_n}{n}<\frac{2n-1}{n} \big)\,.
\end{align*}
Note that, for any integer $0\leq t\leq 2n-1$, we have
$\left\lceil \frac{n-1}{n}t\right\rceil =
 t-\chi(t\ge n)$. In view of that $0\leq e_n\leq 2n-1$, it is easy to observe that
\begin{itemize}
  \item[(1)] $\frac{e_{n-1}}{n-1} < \frac{e_n}{n}$ if and only if $e_{n-1}< e_{n}-\chi \left(e_n\ge n\right)$;
  \item[(2)] $\frac{e_{n-1}}{n-1}+\frac{e_n}{n}<\frac{2n-1}{n}$  if and only if $e_{n-1} + e_{n} < 2n-2+\chi \left(e_n\ge n \right)$.
\end{itemize}
Therefore, for $n\ge 3$ and $\e=(e_{1},\ldots,e_{n})\in\I_{n}^{(2,4,\dots)}$ with $e_n< n$,
 we have
\begin{align*}
\widetilde{\asc}_{D}\,\e
= \  & {\asc}_{D}\,(e_{1},\ldots,e_{n-1})
+ \chi \left(e_{n-1} < e_{n} \right)
+ \chi \left(e_{n-1} < 2n-e_{n}-2 \right).
\end{align*}
Hence, for $n\geq3$ and $0\le i \le n-1$, we get
\begin{align*}
\widetilde{T}_{n,i}(x)=\sum_{j=0}^{2n-3} x^{\chi\left(j<i\right)+\chi\left(j<2n-i-2\right)}T_{n-1,j}(x),
\end{align*}
which is just the first equality.

We proceed to prove the seconde equality. Note that for $n\ge 3$ and $\e=(e_{1},\ldots,e_{n})\in\I_{n}^{(2,4,\dots)}$ with $e_n \ge n$,
 we have
\begin{align*}
\widetilde{\asc}_{D}\,\e
= \  & {\asc}_{D}\,(e_{1},\ldots,e_{n-1})
+ \chi(e_{n-1} < e_{n} -1)
+ \chi(e_{n-1} < 2n-e_{n}-1).
\end{align*}
Thus, for any $0\leq i\leq n-1$, we have
\begin{align*}
\widetilde{T}_{n,2n-i-1}(x) & = \sum_{j=0}^{2n-3} x^{\chi\left(j<(2n-i-1)-1\right)+\chi\left(j<2n-1-(2n-i-1)\right)}T_{n-1,j}(x)\\[5pt]
& = \sum_{j=0}^{2n-3} x^{\chi\left(j<2n-i-2\right)+\chi\left(j<i\right)}T_{n-1,j}(x)\\[5pt]
&= \widetilde{T}_{n,i}(x),
\end{align*}
as desired. This completes the proof.
\end{proof}

By the above lemma, we obtain an expression of $\widetilde{D}_{n}(x)$ in terms of the polynomials $T_{n-1,i}$, which is essential for our proof of the real-rootedness of $\widetilde{D}_{n}(x)$.

\begin{prop}
For $n\ge4$,
\begin{align}\label{recurrence relation Weyl D}
\widetilde{D}_{n}(x)= \sum_{i=0}^{n-2} \big( (n-i-1)x+i+1 \big) \big( xT_{n-1,i}(x)+T_{n-1,n+i-1}(x) \big).
\end{align}
\end{prop}

\begin{proof}
By \eqref{eq:widetilde-T-dual}, we have
\begin{align*}
\widetilde{T}_n(x) \ = & \ \sum_{i=0}^{2n-1}T_{n,i}(x;q) = \ 2 \sum_{i=0}^{n-1}T_{n,i}(x;q).
\end{align*}
Noting that $\widetilde{T}_n(x) = 2 \widetilde{D}_{n}(x)$, it follows that
\begin{align*}
\widetilde{D}_{n}(x) = \sum_{i=0}^{n-1}\widetilde{T}_{n,i}(x).
\end{align*}
\allowdisplaybreaks
Furthermore, by \eqref{eq:widetilde-T-recurrence}, we get
\begin{align*}
\widetilde{D}_{n}(x) = & \ \sum_{i=0}^{n-1} \Big(x^2 \sum_{j=0}^{i-1} T_{n-1,j}(x) + x \sum_{j=i}^{2n-i-3} T_{n-1,j}(x) + \sum_{j=2n-i-2}^{2n-3}T_{n-1,j}(x) \Big) \\[5pt]
= & \ x^2 \sum_{i=0}^{n-1} \sum_{j=0}^{i-1} T_{n-1,j}(x) + x \sum_{i=0}^{n-1} \sum_{j=i}^{2n-i-3} T_{n-1,j}(x) + \sum_{i=0}^{n-1} \sum_{j=2n-i-2}^{2n-3}T_{n-1,j}(x)  \\[5pt]
= & \ x^2 \sum_{i=0}^{n-1} \sum_{j=0}^{i-1} T_{n-1,j}(x) + x \sum_{i=0}^{n-1} \big(\sum_{j=i}^{n-2} + \sum_{j=n-1}^{2n-i-3}\big) T_{n-1,j}(x) + \sum_{i=0}^{n-1} \sum_{j=2n-i-2}^{2n-3}T_{n-1,j}(x)  \\[5pt]
= & \ x^2 \sum_{i=0}^{n-1} \sum_{j=0}^{i-1} T_{n-1,j}(x) + x \sum_{i=0}^{n-1} \sum_{j=i}^{n-2}T_{n-1,j}(x)  \\[5pt]
&\, + x \sum_{i=0}^{n-1}\sum_{j=n-1}^{2n-i-3} T_{n-1,j}(x) + \sum_{i=0}^{n-1} \sum_{j=2n-i-2}^{2n-3}T_{n-1,j}(x)  \\[5pt]
= & \ x^2 \sum_{i=0}^{n-1} \sum_{j=0}^{i-1} T_{n-1,j}(x) + x \sum_{i=0}^{n-1} \sum_{j=i}^{n-2}T_{n-1,j}(x)  \\[5pt]
&\, + x \sum_{i=0}^{n-1}\sum_{j=0}^{n-i-2} T_{n-1,n+j-1}(x) + \sum_{i=0}^{n-1} \sum_{j=n-i-1}^{n-2}T_{n-1,n+j-1}(x).
\end{align*}
Then, for each double summation, we interchange the order of summation,
\begin{align*}
\widetilde{D}_{n}(x) = & \ x^2 \sum_{j=0}^{n-2} \sum_{i=j+1}^{n-1} T_{n-1,j}(x) + x
\sum_{j=0}^{n-2} \sum_{i=0}^{j}
T_{n-1,j}(x)  \\[5pt]
&\, + x \sum_{j=0}^{n-2}\sum_{i=0}^{n-j-2}
 T_{n-1,n+j-1}(x) + \sum_{j=0}^{n-2} \sum_{i=n-j-1}^{n-1}T_{n-1,n+j-1}(x)  \\[5pt]
 =&  \ x^2 \sum_{j=0}^{n-2}  (n-j-1) T_{n-1,j}(x) + x \sum_{j=0}^{n-2} (j+1) T_{n-1,j}(x) \\
    &  +x  \sum_{j=0}^{n-2} (n-j-1) T_{n-1,n+j-1}(x) + \sum_{j=0}^{n-2} (j+1) T_{n-1,n+j-1}(x),
\end{align*}
as desired. This completes the proof.
\end{proof}

In the following we shall prove that the sequence $(xT_{n-1,i}(x)+T_{n-1,n+i-1}(x))_{i=0}^{n-2}$ is mutually interlacing.
To this end, define a sequence of polynomials $(K_{n,i}(x))_{i=0}^{2n-1}$
in the following way:
\begin{align}\label{eq:f}
K_{n,i}(x) & \ = \left\{
\begin{array}{ll}
T_{n,i}(x)+T_{n,n+i}(x), & \mbox{ if }0\leq i\leq n-1,\\[5pt]
xT_{n,i-n}(x)+T_{n,i}(x), & \mbox{ if }n\leq i\leq 2n-1.
\end{array}
\right.
\end{align}
Note that $(K_{n,i}(x))_{i=n}^{2n-1}$ is just the sequence $(xT_{n,i}(x)+T_{n,n+i}(x))_{i=0}^{n-1}$. Instead of proving the mutual interlacing of $(K_{n,i}(x))_{i=n}^{2n-1}$, we directly prove the mutual interlacing of the entire sequence $(K_{n,i}(x))_{i=0}^{2n-1}$.
A remarkable property of this sequence is that it satisfies the same recurrence relation as  $(T_{n,i}(x))_{i=0}^{2n-1}$.

\begin{prop}
For $n\geq3$ and $0\le i\le 2n-1$, we have
\begin{align}\label{f-recurrence}
K_{n,i}(x) \ = \ x\sum_{j=0}^{\left\lceil \frac{n-1}{n}i\right\rceil -1}K_{n-1,j}(x)+\sum_{j=\left\lceil \frac{n-1}{n}i\right\rceil }^{2n-3}K_{n-1,j}(x)\,.
\end{align}
\end{prop}

\begin{proof}
We use some matrix techniques to give a proof. For notational convenience, we shall consider $(K_{n,i}(x))_{i=0}^{2n-1}$ and $(T_{n,i}(x))_{i=0}^{2n-1}$ as column vectors.
From  \eqref{eq:f} it follows that
\begin{align}\label{f-recurrence-1}
 \big(K_{n,i}(x)\big)_{i=0}^{2n-1} & = \left(\begin{array}{cc}
  I_n & I_n \\
 xI_n & I_n
\end{array}\right)
\big(T_{n,i}(x)\big)_{i=0}^{2n-1},
\end{align}
where $I_n$ is the identity matrix of order $n$.
Note that the recurrence relation \eqref{eq:T_n,i} can be rewritten as
\begin{align*}
\begin{split}
T_{n,i}(x)    &  = x\ \sum_{j=0}^{i-1}T_{n-1,j}(x)+ \ \sum_{j=i}^{2n-3}T_{n-1,j}(x),\\
T_{n,n+i}(x)  &  = x \sum_{j=0}^{n+i-2}T_{n-1,j}(x)+ \sum_{j=n+i-1}^{2n-3}T_{n-1,j}(x),
\end{split}
\end{align*}
where $0 \le i \le n-1$. Therefore, we get
\begin{align}\label{T-recurrence-matrix}
\big(T_{n,i}(x)\big)_{i=0}^{2n-1} = \left(\begin{array}{cc}
  A & B \\
 xB & A
\end{array}\right)
\big(T_{n-1,i}(x)\big)_{i=0}^{2n-3},
\end{align}
where \begin{align*}
A = \left(\begin{array}{cccc}
  1 & 1 & \cdots & 1 \\
  x & 1 & \cdots & 1 \\
  x & x & \cdots & 1 \\
  \cdots & \cdots & \cdots & \cdots \\
  x & x & \cdots & x \\
\end{array}\right)_{(n-1)\times(n-2)}
\end{align*}
and
\begin{align*}
B = \left(\begin{array}{cccc}
  1 & 1 & \cdots & 1 \\
  1 & 1 & \cdots & 1 \\
  1 & 1 & \cdots & 1 \\
  \cdots & \cdots & \cdots & \cdots \\
  1 & 1 & \cdots & 1 \\
\end{array}\right)_{(n-1)\times(n-2)}.
\end{align*}
One can compute that
\begin{align}\label{inter-eq}
\left(\begin{array}{cc}
  I_n & I_n \\
 xI_n & I_n
\end{array}\right)
\left(\begin{array}{cc}
  A & B \\
 xB & A
\end{array}\right)& =
\left(\begin{array}{cc}
  A+xB & A+B \\
 xA+xB & A+xB
\end{array}\right)\nonumber\\[8pt]
& =
\left(\begin{array}{cc}
  A & B \\
 xB & A
\end{array}\right)
\left(\begin{array}{cc}
  I_{n-1} & I_{n-1} \\
 xI_{n-1} & I_{n-1}
\end{array}\right).
\end{align}

Combing \eqref{f-recurrence-1}, \eqref{T-recurrence-matrix} and \eqref{inter-eq}, we obtain
\begin{align*}
\big(K_{n,i}(x)\big)_{i=0}^{2n-1} & = \left(\begin{array}{cc}
  I_n & I_n \\
 xI_n & I_n
\end{array}\right)
\big(T_{n,i}(x)\big)_{i=0}^{2n-1}\\
 & =  \left(\begin{array}{cc}
  I_n & I_n \\
 xI_n & I_n
\end{array}\right)
\left(\begin{array}{cc}
  A & B \\
 xB & A
\end{array}\right)
\big(T_{n-1,i}(x)\big)_{i=0}^{2n-3},\\
 & =  \left(\begin{array}{cc}
  A & B \\
 xB & A
\end{array}\right)
\left(\begin{array}{cc}
  I_{n-1} & I_{n-1} \\
 xI_{n-1} & I_{n-1}
\end{array}\right)
\big(T_{n-1,i}(x)\big)_{i=0}^{2n-3}\\
& = \left(\begin{array}{cc}
  A & B \\
 xB & A
\end{array}\right)
\big(K_{n-1,i}(x)\big)_{i=0}^{2n-3},
\end{align*}
which is equivalent to \eqref{f-recurrence}. This completes the proof.
\end{proof}

Now we can prove the mutually interlacing properpty of $(K_{n,i}(x))_{i=0}^{2n-1}$.

\begin{prop}
\label{f mutually interlacing} For $n\geq4$, the sequence
of polynomials $(K_{n,i}(x))_{i=0}^{2n-1}$ is mutually interlacing.
\end{prop}

\begin{proof} We use induction on $n$. For $n=4$,  by using \eqref{eq:f}, we can directly compute the polynomials $K_{4,i}(x)$ for $0\leq i\leq 7$. The eight polynomials are listed below together with the values of their zeros rounded to 4 significant figures :
$$
\begin{array}{ll}
K_{4,0}(x)=12x^3+50x^2+32x+2,&\{-3.396,-0.7008,-0.07004\},\\[5pt]
K_{4,1}(x)=18x^3+52x^2+26x,&\{-2.246,-0.6432,0\},\\[5pt]
K_{4,2}(x)=26x^3+52x^2+18x,&\{-1.555,-0.4453,0\},\\[5pt]
K_{4,3}(x)=2x^4+32x^3+50x^2+12x,&\{-14.28,-1.427,-0.2945,0\},\\[5pt]
K_{4,4}(x)=2x^4+32x^3+50x^2+12x,&\{-14.28,-1.427,-0.2945,0\},\\[5pt]
K_{4,5}(x)=4x^4+38x^3+48x^2+6x,&\{-8.029,-1.331,-0.1404,0\},\\[5pt]
K_{4,6}(x)=6x^4+48x^3+38x^2+4x,&\{-7.124,-0.7513,-0.1246,0\},\\[5pt]
K_{4,7}(x)=12x^4+50x^3+32x^2+2x,&\{-3.396,-0.7008,-0.07004,0\}.
\end{array}
$$
One can easily verify that the polynomial $K_{4,i}$ interlaces $K_{4,j}$ for any $0\le i<j \le 7$. Then, by Theorem \ref{thm:interlacing}, we obtain the mutually interlacing property of $(K_{n,i}(x))_{i=0}^{2n-1}$ for any $n\geq 4$. This completes the proof.
\end{proof}

Now we are in position to prove the real-rootedness of $\widetilde{D}_{n}(x)$.

\begin{thm}
\label{Weyl D_n real}
For $n\ge3$, the polynomial $\widetilde{D}_{n}(x)$
has only real zeros. \end{thm}

\begin{proof}
Clearly, $\widetilde{D}_{3}(x)=4 x + 16 x^2 + 4 x^3$ has only real zeros.

By \eqref{recurrence relation Weyl D}, it suffices to prove the real-rootedness of the following polynomial
\begin{align*}
\sum_{i=1}^{n-1} \big( (n-i)x+i \big) K_{n-1,n+i-2}(x)
\end{align*}
for any $n\ge4$.

By Proposition \ref{f mutually interlacing}, we know that, for any $n\ge4$,
the sequence $(K_{n-1,i}(x))_{i=n-1}^{2n-3}$,
as a subsequence of $(K_{n-1,i}(x))_{i=0}^{2n-3}$, is mutually interlacing.
Let $m=n-1$ and define $a_i=n-i$, $b_i=i$ and $f_{i}(x)=K_{n-1,n+i-2}(x)$ for $1\le i \le n-1$ in Theorem \ref{Haglund1999}. Since $a_{i}b_{i+1}- b_{i}a_{i+1}=n>0$, it is immediate that
$$\sum_{i=1}^{n-1} (n-i) K_{n-1,n+i-2}(x) \ \sep \ \sum_{i=1}^{n-1} i \, K_{n-1,n+i-2}(x).$$
Since all the zeros of these two polynomials are real and nonpositive, we get $$\sum_{i=1}^{n-1} i K_{n-1,n+i-2}(x) \ \sep \ x\sum_{i=1}^{n-1} (n-i) K_{n-1,n+i-2}(x).$$
Further, by Lemma \ref{lem:roots-interlacing}, we obtain the desired result. This completes the proof.
\end{proof}

Finally, we can prove Conjecture \ref{conj:affine-type-D}. Combining Theorem \ref{Weyl D_n real} and the known results
on the real-rootedness of affine Eulerian polynomials of other types,
we obtain the following result, which gives a complete answer to
Dilks, Petersen, and Stembridge's conjecture.

\begin{thm}
For any irreducible finite Weyl group $W$, the affine Eulerian polynomial $\widetilde{W}(x)$ has only real zeros.
\end{thm}

As we have seen that, the decomposition of {\textup{$\widetilde{D}_{n}(x)$}} given in \eqref{recurrence relation Weyl D} plays an important role in our proof of Dilks, Petersen, and Stembridge's conjecture. In the following, we shall use this decomposition to prove a relation concerning $\widetilde{B}_{n}(x),\,\widetilde{D}_{n}(x)$
and $D_{n-1}(x)$, which was established by Dilks, Petersen, and Stembridge.

\begin{prop}
[{\cite[Proposition 6.2]{Dilks2009Affine}}] For $n\geq3$, we have
\begin{align}\label{Dilks-6.2}
\widetilde{D}_{n}(x)=\widetilde{B}_{n}(x)-2nx\,D_{n-1}(x).
\end{align}
Moreover, $\widetilde{B}_{n}(x) \sep \widetilde{B}_{n+1}(x)$,
$D_{n}(x) \sep D_{n+1}(x)$
and $D_{n}(x) \sep \widetilde{B}_{n}(x)$.
\end{prop}

\begin{proof}
Let us first prove \eqref{Dilks-6.2}.
By \eqref{recurrence relation Weyl D}, we know that $\widetilde{D}_{n}(x)$ can be expressed in terms of $T_{n-1,i}$. We attempt to expand the other two  polynomials $\widetilde{B}_{n}(x)$ and $D_{n-1}(x)$ as well and then prove these expansions satisfy the desired equality.

As noted by Savage and Visontai \cite{Savage$s$}, $\widetilde{B}_{n}(x)=T_{n+1,n+1}(x)$.  By the recurrence relation \eqref{eq:T_n,i}, we obtain that
$$
\widetilde{B}_{n}(x)  =
\sum_{i=0}^{n-1}\big(x\,T_{n,i}(x)+T_{n,n+i}(x)\big).
$$
\allowdisplaybreaks
Further, by \eqref{eq:T_n,i}, we have
\begin{align*}
\widetilde{B}_{n}(x) =
& \ \sum_{i=0}^{n-1}\Big(x\big( x\ \sum_{j=0}^{i-1}T_{n-1,j}(x)+ \sum_{j=i}^{2n-3}T_{n-1,j}(x)\big)\\[5pt]
& +\big(x \sum_{j=0}^{n+i-2}T_{n-1,j}(x)+ \sum_{j=n+i-1}^{2n-3}T_{n-1,j}(x)\big)\Big)\\[5pt]
=& \ \sum_{i=0}^{n-1}\Big(x\big( x\ \sum_{j=0}^{i-1}T_{n-1,j}(x)+ (\sum_{j=i}^{n-2}+\sum_{j=n-1}^{2n-3})T_{n-1,j}(x)\big)\\[5pt]
& +\big(x (\sum_{j=0}^{n-2}+\sum_{j=n-1}^{n+i-2}) T_{n-1,j}(x)+ \sum_{j=n+i-1}^{2n-3}T_{n-1,j}(x)\big)\Big)\\[5pt]
=& \ x^2\sum_{i=0}^{n-1}\sum_{j=0}^{i-1}T_{n-1,j}(x) + x\sum_{i=0}^{n-1}\sum_{j=i}^{n-2}T_{n-1,j}(x) +x\sum_{i=0}^{n-1}\sum_{j=n-1}^{2n-3}T_{n-1,j}(x)\\[5pt]
& +x \sum_{i=0}^{n-1}\sum_{j=0}^{n-2}T_{n-1,j}(x)+x \sum_{i=0}^{n-1} \sum_{j=n-1}^{n+i-2} T_{n-1,j}(x)+ \sum_{i=0}^{n-1}\sum_{j=n+i-1}^{2n-3}T_{n-1,j}(x)\\[5pt]
=& \ x^2\sum_{i=0}^{n-1}\sum_{j=0}^{i-1}T_{n-1,j}(x) + x\sum_{i=0}^{n-1}\sum_{j=i}^{n-2}T_{n-1,j}(x) +x\sum_{i=0}^{n-1}\sum_{j=0}^{n-2}T_{n-1,n+j-1}(x)\\[5pt]
& +x \sum_{i=0}^{n-1}\sum_{j=0}^{n-2}T_{n-1,j}(x)+x\sum_{i=0}^{n-1} \sum_{j=0}^{i-1} T_{n-1,n+j-1}(x)+ \sum_{i=0}^{n-1}\sum_{j=i}^{n-2}T_{n-1,n+j-1}(x).
\end{align*}
Then, for each double summation, we interchange the order of summation and get
\begin{align*}
\sum_{i=0}^{n-1}\sum_{j=0}^{i-1}T_{n-1,j}(x) &  =\ \sum_{j=0}^{n-2}\sum_{i=j+1}^{n-1}T_{n-1,j}(x) = \sum_{j=0}^{n-2}(n-j-1)T_{n-1,j}(x),\\[5pt]
\sum_{i=0}^{n-1}\sum_{j=i}^{n-2}T_{n-1,j}(x) & = \ \sum_{j=0}^{n-2}\sum_{i=0}^{j}T_{n-1,j}(x) = \sum_{j=0}^{n-2}(j+1)T_{n-1,j}(x),\\[5pt]
\sum_{i=0}^{n-1}\sum_{j=0}^{n-2}T_{n-1,n+j-1}(x) &  = \ \sum_{j=0}^{n-2}\sum_{i=0}^{n-1}T_{n-1,n+j-1}(x)  = \sum_{j=0}^{n-2}n\,T_{n-1,n+j-1}(x),\\[5pt]
\sum_{i=0}^{n-1}\sum_{j=0}^{n-2}T_{n-1,j}(x) &  = \  \sum_{j=0}^{n-2}\sum_{i=0}^{n-1}T_{n-1,j}(x)  =  \sum_{j=0}^{n-2}n\,T_{n-1,j}(x),\\[5pt]
\sum_{i=0}^{n-1} \sum_{j=0}^{i-1} T_{n-1,n+j-1}(x) &  = \ \sum_{j=0}^{n-2}
\sum_{i=j+1}^{n-1} T_{n-1,n+j-1}(x)  = \sum_{j=0}^{n-2}(n-j-1)T_{n-1,n+j-1}(x),\\[5pt]
\sum_{i=0}^{n-1}\sum_{j=i}^{n-2}T_{n-1,n+j-1}(x)  &  = \ \sum_{j=0}^{n-2}\sum_{i=0}^{j}T_{n-1,n+j-1}(x)  =  \sum_{j=0}^{n-2}(j+1)T_{n-1,n+j-1}(x).
\end{align*}
Therefore, it follows that
\begin{equation}\label{eq:1}
\begin{aligned}
\widetilde{B}_{n}(x) = &
   \sum_{j=0}^{n-2}\big((n+j-1)x+n+j+1\big)x\,T_{n-1,j}(x) \\
 & \ +\sum_{j=0}^{n-2}\big((2n-j-1)x+j+1\big)T_{n-1,n+j-1}.
\end{aligned}
\end{equation}
On the other hand, by \eqref{eq:T_n,i}, we have
\begin{align}\label{eq:2}
D_{n-1}(x) = \frac{1}{2}T_{n,0}(x)
= \frac{1}{2}\sum_{i=0}^{n-2}\big(T_{n-1,i}(x)+T_{n-1,n+i-1}(x)\big).
\end{align}
Combining \eqref{eq:1} and \eqref{eq:2}, we get
\begin{align*}
\widetilde{B}_{n}(x)-2nxD_{n-1}(x) =&\sum_{i=0}^{n-2}\big((n+i-1)x+n+i+1\big)xT_{n-1,i}(x)\nonumber\\
   & +\sum_{i=0}^{n-2}\big((2n-i-1)x+i+1\big)T_{n-1,n+i-1}\\
   & -\sum_{i=0}^{n-2}nx\big(T_{n-1,i}(x)+T_{n-1,n+i-1}(x)\big)\\
=& \sum_{i=0}^{n-2} \big( (n-i-1)x+i+1 \big) \big( xT_{n-1,i}(x)+T_{n-1,n+i-1}(x) \big),
\end{align*}
which is equal to $\widetilde{D}_{n}(x)$ by \eqref{recurrence relation Weyl D}.

 We proceed to prove the rest of the assertions. By the mutually interlacing property of $\big( T_{n,i}(x)\big)_{i=0}^{2n-1}$ and the fact that these polynomials have only nonnegative coefficients for $0 \le i \le n-1$, we get that
\begin{align*}
 T_{n,n}(x) \ \sep \ x \,T_{n,i}(x) \quad \mbox{and} \quad T_{n,n}(x) \ \sep \  T_{n,n+i}(x).
\end{align*}
Then, by Lemma \ref{lem:roots-interlacing},
we obtain that
\begin{align*}
\widetilde{B}_{n-1}(x) = T_{n,n}(x)  \ \sep \
 \sum_{i=0}^{n-1}\big(xT_{n,i}(x)+T_{n,n+i}(x)\big) = \widetilde{B}_{n}(x).
\end{align*}

Similarly, since for any $0\leq i\leq n-1$ it holds
\begin{align*}
T_{n,0}(x)  \ \sep \  T_{n,i}(x) \mbox{ and} \ T_{n,0}(x) \ \sep \  T_{n,n+i}(x),
\end{align*}
we obtain that
\begin{align*}
D_{n-1}(x) = \dfrac{1}{2}T_{n,0}(x)
 \ \sep \  \dfrac{1}{2}\sum_{i=0}^{n-1}\big(T_{n,i}(x)+T_{n,n-1+i}(x)\big) = D_{n}(x).
\end{align*}

Finally, for the interlacing relation between $D_{n}(x)$
and $\widetilde{B}_{n}(x)$, we have
\begin{align*}
D_{n}(x) = \dfrac{1}{2}T_{n+1,0}(x) \ \sep \  T_{n+1,n+1}(x)=\widetilde{B}_{n}(x).
\end{align*}
This completes the proof.
\end{proof}

\begin{rem}
The descent polynomials for Coxeter groups have a similar
property, which states that
\begin{align*}
D_{n}(x)=B_{n}(x)-n2^{n-1}x\,A_{n-2}(x).
\end{align*}
This was found by Stembridge \cite[Lemma 9.1]{Stembridge1994Some}.
For more information, see \cite[Theorem 4.7]{Brenti1994$q$} and \cite[Proposition 6.3]{Dilks2009Affine}.
\end{rem}

\vskip 3mm
\noindent {\bf Acknowledgments.} This work was supported by the 973 Project and the National Science Foundation of China.

\end{document}